\documentclass[a4paper]{article}
\usepackage{mathptmx}
\usepackage{amscd,amssymb,amsmath,amsthm}

\makeatletter
\renewcommand{\@seccntformat}[1]
{{\csname the#1\endcsname}.\hspace{0.3em}}

\renewcommand{\section}{\@startsection
{section}
{1}
{0mm}
{-1.5\baselineskip}
{\baselineskip}
{\bfseries\normalsize}}

\renewcommand{\subsection}{\@startsection
{subsection}
{2}
{0mm}
{-\baselineskip}
{0.5\baselineskip}
{\normalsize\itshape}}

\makeatother

\theoremstyle{plain}

\newtheorem{lemma}{Lemma}
\newtheorem{corollary}[lemma]{Corollary}
\newtheorem{prop}[lemma]{Proposition}

\newtheorem*{TA}{Theorem~A}
\newtheorem*{TB}{Theorem~B}
\newtheorem*{HL}{Hersch's Lemma}

\theoremstyle{definition}
\newtheorem*{defin*}{Definition}

\theoremstyle{remark}

\newtheorem*{example*}{Example}
\newtheorem*{remark*}{Remark}

\DeclareMathAlphabet{\matheur}{U}{eur}{m}{n}
\DeclareMathAlphabet{\matheus}{U}{eus}{m}{n}
\DeclareMathAlphabet{\matheuf}{U}{euf}{m}{n}

\numberwithin{equation}{section}
\newcommand{\abs}[1]{\left\lvert#1\right\rvert}

\DeclareMathOperator{\PSL}{PSL}

\begin{document}

\author{Gerasim  Kokarev
\\ {\small\it School of Mathematics, The University of Edinburgh}
\\ {\small\it King's Buildings, Mayfield Road, Edinburgh EH9 3JZ, UK}
\\ {\small\it Email: {\tt G.Kokarev@ed.ac.uk}}
}

\title{On the concentration-compactness phenomenon for the first
  Schrodinger eigenvalue}
\date{}
\maketitle

\begin{abstract}
We study a variational problem for the first eigenvalue $\lambda_1(V)$ of
the Schrodinger operator $(-\Delta_g+V)$ on closed Riemannian
surfaces. More precisely, we explore concentration-compactness
properties of sequences formed by $\lambda_1$-extremal potentials.
\end{abstract}

%


\section{Introduction}
Let $M$ be a closed smooth surface endowed with a Riemannian metric
$g$ of volume one. For a function $V\in L^\infty(M)$ we denote by 
$$
\lambda_0(V)<\lambda_1(V)\leqslant\ldots\leqslant\lambda_k(V)\leqslant\ldots
$$
eigenvalues of the Schrodinger operator $(-\Delta_g+V)$. We suppose
that potentials $V$ have zero mean-values and denote their space by
$$
L^\infty_0 (M)=\{ V\in L^\infty(M): \int_M V\mathit{dVol}_g=0\}.
$$
It is a simple exercise to see that the zero eigenvalue $\lambda_0(V)$
is non-positive for any $V\in L^\infty_0(M)$. The next eigenvalue
$\lambda_1(V)$ is also bounded from above when $V$ ranges over
$L^\infty_0(M)$. More precisely, by the work of Li and
Yau~\cite{LY82}, see also~\cite{El92}, the first eigenvalue
$\lambda_1(V)$ can be estimated by the conformal volume; the latter
can be bounded in terms of the genus $\gamma$ of $M$ in many
cases. For example, for an orientable surface $M$ we have
$$
\lambda_1(V)\leqslant 8\pi(\gamma+1),
$$
where $V$ ranges over $L^\infty_0(M)$. 

We regard the first eigenvalue $\lambda_1(V)$ as a functional on the
space of potentials $L^\infty_0(M)$, and are interested in its
critical points. Following Nadirashvili~\cite{Na96}, a potential $V$
is called $\lambda_1$-{\it extremal} if for any $q\in L^\infty_0(M)$
the function $\lambda_1(V+tq)$, where $t$ ranges in a neighbourhood of
zero, satisfies either the inequality
$$
\lambda_1(V+tq)\leqslant\lambda_1(V)+o(t)\qquad\text{as}\quad t\to 0, 
$$
or the inequality
$$
\lambda_1(V+tq)\geqslant\lambda_1(V)+o(t)\qquad\text{as}\quad t\to 0.
$$
In particular, we see that any $\lambda_1$-{\it maximal} potential is
extremal. Basic properties of extremal potentials are discussed in
Sect.~\ref{prems}. Here we mention only that any extremal potential
$V$ is $C^\infty$-smooth.

The purpose of this note is to communicate the following
concentration-compactness alternative for the first eigenvalue
functional.
\begin{TA}
Let $M$ be a closed surface endowed with a Riemannian metric $g$ of
volume one and $V_n\in L^\infty_0(M)$, $n=1,2,\ldots$ , be a sequence
of $\lambda_1$-extremal potentials for the Schrodinger operator
$(-\Delta_g+V)$. Then there exists a subsequence $V_{n_k}$ such that
one of the following holds:
\begin{itemize}
\item[(i)] the subsequence $(V_{n_k})$ converges in the sense of
  distributions to $8\pi(1-\delta_x)$, for some $x\in M$, and
  $\lambda_1(V_{n_k})\to 8\pi$ as $n_k\to +\infty$;
\item[(ii)] the subsequence $(V_{n_k})$ converges in
  $C^\infty$-topology to an extremal potential $V\in L^\infty_0(M)$,
  and $\lambda_1(V_{n_k})\to\lambda_1(V)>0$ as $n_k\to +\infty$.
\end{itemize}
\end{TA}

As a consequence, we see that the set of extremal potentials whose
first eigenvalues are bounded away from $8\pi$ is always compact. The
critical value $8\pi$ is the maximal first eigenvalue on the
$2$-sphere endowed with the standard round metric of volume one, and
by Prop.~\ref{sphere} below maximal potentials on it form a
non-compact space. 

The proof of the alternative is based on the two ingredients:
characterisation of extremal potentials via harmonic maps and the use
of the bubbling convergence theorem for the latter. The proof that the
critical value for the concentration can be only $8\pi$ involves more
detailed study of the Schrodinger equation near the bubble point. In
the process we obtain a general upper estimate (Lemma~\ref{upper_b})
for the critical value of a sequence of (not necessarily extremal!)
concentrating potentials. Our methods, combined with the version of
the bubbling convergence theorem in~\cite{Pa96}, also yield a version
of the result for the case of a variable metric on $M$. We state it
explicitly for the completeness.
\begin{TB}
Let $g_n$, $n=1,2,\ldots$ , be a sequence of unit volume metrics on
$M$ converging in $C^2$-topology to a metric $g$, and $V_n$,
$n=1,2,\ldots$ , be a sequence of potentials such that each $V_n$ is
$\lambda_1$-extremal for the Schrodinger operator
$(-\Delta_{g_n}+V)$. Then there exists a subsequence $(V_{n_k})$ such
that one of the following holds:
\begin{itemize}
\item[(i)] the subsequence $(V_{n_k})$ converges in the sense of
  distributions to $8\pi(1-\delta_x)$, for some $x\in M$, and
  $\lambda_1(V_{n_k})\to 8\pi$ as $n_k\to +\infty$;
\item[(ii)] the subsequence $(V_{n_k})$ converges in
  $C^\infty$-topology to an extremal potential $V\in L^\infty_0(M)$,
  and $\lambda_1(V_{n_k})\to\lambda_1(V)>0$ as $n_k\to +\infty$.
\end{itemize}
\end{TB}

\noindent
In a forthcoming paper we shall study questions related to this
concentration-compactness alternative in dimension greater than two.

\smallskip
\noindent {\it Acknowledgments.} The author is grateful to Nikolai
Nadirashvili for a number of discussions on the subject. The author
acknowledges the support of EPSRC and the Maxwell Mathematical
Institute during the work on the paper.

\section{Preliminaries}
\label{prems}
\subsection{Notation. Properties of extremal potentials}
Let $M$ be a closed smooth surface. For a Riemannian metric $g$ on $M$
the Laplace operator $\Delta_g$ in local coordinates $(x^i)$,
$1\leqslant i\leqslant 2$, has the form
$$
\Delta_g=\frac{1}{\sqrt{\abs{g}}}\frac{\partial~}{\partial x^i}
\left(\sqrt{\abs{g}}g^{ij}\frac{\partial~}{\partial x^j}\right),
$$
where $(g_{ij})$ are components of the metric $g$, $(g^{ij})$ is the
inverse tensor, and $\abs{g}$ stands for $\det(g_{ij})$. We use above
the summation convention for the repeated indices. We suppose
throughout that $g$ is normalised such that $\mathit{Vol}_g(M)$ equals
one. For a function $V\in L^\infty_0(M)$ we denote by
$$
\lambda_0(V)<\lambda_1(V)\leqslant\ldots\leqslant\lambda_k(V)\leqslant\ldots
$$
the eigenvalues of the Schrodinger operator $(-\Delta_g+V)$; these are
real numbers for which the equation
\begin{equation}
\label{eifun}
(-\Delta_g+V)u=\lambda_ku
\end{equation}
has a non-trivial solution. The solutions of equation~\eqref{eifun} are
called eigenfunctions. Recall that by variational characterisation the
eigenvalue $\lambda_k(V)$ is the infimum of the Rayleigh quotient 
$$
\matheur R_V(u)=\frac{\int_M\abs{\nabla u}^2
\mathit{dVol}_g+\int_MVu^2\mathit{dVol}_g}{\int_Mu^2\mathit{dVol}_g}
$$
over the set of all smooth functions $u$ that are $L_2$-orthogonal to
the eigenfunctions for $\lambda_0$, $\lambda_1$, \ldots,
$\lambda_{k-1}$. In particular, we see that
$$
\lambda_0(V)\leqslant\matheur R_V(1)=0\quad\text{for any }V\in
L^\infty_0(M).
$$

Now we discuss the properties of extremal potentials. The following
proposition is a sharpened version of the one due to
Nadirashvili~\cite{Na96}; below we outline the proof based on the
first variation formula for $\lambda_1$.
\begin{prop}
\label{props}
For a function $V\in L^\infty_0(M)$ the following hypotheses are
equivalent:
\begin{itemize}
\item[(i)] $V$ is a $\lambda_1$-extremal potential for the Schrodinger
  operator;
\item[(ii)] the quadratic form
\begin{equation}
\label{form}
u\longmapsto \int_Mqu^2\mathit{dVol}_g
\end{equation}
is indefinite on the space of $\lambda_1$-eigenfunctions of the
Schrodinger operator $(-\Delta_g+V)$ for any $q\in
L^\infty_0(M)$;
\item[(iii)] there exists a finite collection of
  $\lambda_1$-eigenfunctions $(u_i)$ of the Schrodinger operator
  $(-\Delta_g+V)$ such that $\sum_i u_i^2=1.$
\end{itemize}
\end{prop}
\begin{proof}[Outline of the proof.]
Denote by $V_t$ the family of functions $(V+tq)$ in $L^\infty_0(M)$;
we assume that $t$ ranges in a neighbourhood of zero. Suppose that the
first eigenvalue $\lambda_1(V)$ has multiplicity $m$. Then by 
general perturbation theory~\cite{Kato} there exists $m$ analytic
families of real numbers $\Lambda_{i,t}$ and functions $u_{i,t}$,
where $i=1,\ldots,m$, such that
\begin{equation}
\label{f_1}
(-\Delta_g+V_t)u_{i,t}=\Lambda_{i,t}u_{i,t},\qquad
\Lambda_{i,0}=\lambda_1(V).
\end{equation}
Assume that the $L_2$-norms of the $u_{i,t}$'s are equal to
one. Differentiating relation~\eqref{f_1} with respect to $t$ and
evaluating the result at $t=0$, we obtain
\begin{equation}
\label{f_2}
qu_{i,0}+(-\Delta_g+V)u'_{i,0}=\Lambda'_{i,0}u_{i,0}+\Lambda_{i,0}u'_{i,0}.
\end{equation}
Multiplying the identities~\eqref{f_1} and~\eqref{f_2} by $u'_{i,0}$
and $u_{i,0}$ respectively and substracting the first from the second,
after integration, we obtain the first variation formula 
\begin{equation}
\label{fvf}
\left.\frac{d}{dt}\right|_{t=0}\Lambda_{i,t}=\int_Mqu^2_{i,0}
\mathit{dVol}_g.
\end{equation}
The discussion implies that the function $\lambda_1(V_t)$ is
piece-wise smooth and has left and right derivatives. Moreover, there
exist indices $i$ and $j$ such that
$$
\left.\frac{d}{dt}\right|_{t=0-}\lambda_1(V_t)=\Lambda'_{i,0}\qquad
\text{and}\qquad
\left.\frac{d}{dt}\right|_{t=0+}\lambda_1(V_t)=\Lambda'_{j,0}.
$$
To prove the statement $(i)\Rightarrow (ii)$, we note that the
potential $V$ is extremal if and only if
$$
\left.\frac{d}{dt}\right|_{t=0-}\lambda_1(V_t)\cdot
\left.\frac{d}{dt}\right|_{t=0+}\lambda_1(V_t)\leqslant 0.
$$
This together with relations above and formula~\eqref{fvf} proves that
the form~\eqref{form} is indefinite.

To prove the converse statement $(ii)\Rightarrow (i)$ we note that in
the basis $(u_{i,0})$, where $i=1,\ldots, m$, the quadratic
form~\eqref{form} is diagonal:
$$
\int_Mqu_{i,0}u_{j,0}\mathit{dVol}_g=\left.\frac{d}{dt}\right|_{t=0}
\Lambda_{i,t}\cdot \delta_{ij}.
$$
This follows by differentiating relation~\eqref{f_1} in the manner
similar to the one used to obtain~\eqref{fvf}. Since $\lambda_1(V_t)$
equals $\min_i\Lambda_{i,t}$, we get
$$
\left.\frac{d}{dt}\right|_{t=0+}\lambda_1(V_t)=\min_i\Lambda'_{i,0}=
\min_i\int_Mqu^2_{i,0}\mathit{dVol}_g,
$$
$$
\left.\frac{d}{dt}\right|_{t=0-}\lambda_1(V_t)=\max_i\Lambda'_{i,0}=
\max_i\int_Mqu^2_{i,0}\mathit{dVol}_g.
$$
Since the form~\eqref{form} is indefinite, then either one of this
derivatives vanishes or they have different signs. This means that the
potential $V$ is extremal.

\noindent
$(ii)\Rightarrow (iii)$. Let $K$ be the convex hull of the set of
squared $\lambda_1$-functions $\{u^2:$ $u$ is an
eigenfunction$\}$. Suppose the contrary to the hypotheses~$(iii)$;
then $1\ne K$. By classical separation results, there exists a
function $f\in L_2(M)$ such that
$$
\int_M1\cdot f\mathit{dVol}_g<0\quad\text{and}\quad
\int_M\phi\cdot f\mathit{dVol}_g>0,\quad\text{where }\phi\in
K\backslash\{0\}.
$$
Let $f_0$ be the mean-value part of $f$,
$$
f_0=f-\int_Mf\mathit{dVol}_g.
$$
Then for any eigenfunction $u$ we have
$$
\int_Mf_0u^2\mathit{dVol}_g=\int_Mfu^2\mathit{dVol}_g-\left(\int_Mf
\mathit{dVol}_g\right)\left(\int_Mu^2\mathit{dVol}_g\right)>0.
$$
This is a contradiction with~$(ii)$.

\noindent
$(iii)\Rightarrow (ii)$. Conversely, let $(u_i)$ be a finite
collection of eigenfunctions satisfying the hypothesis~$(iii)$. Then
for any $q\in L^\infty_0(M)$, we have
$$
\int_Mq(\sum_iu_i^2)\mathit{dVol}_g=\int_Mq\mathit{dVol}_g=0.
$$
This demonstrates the hypothesis~$(ii)$.
\end{proof}

As a consequence we point out the following properties of extremal
potentials.
\begin{corollary}
\label{props:col}
Let $V\in L^\infty_0(M)$ be an extremal potential for the Schrodinger
operator. Then $V$ is $C^\infty$-smooth, and is bounded by its first
eigenvalue $\lambda_1(V)\geqslant V$. Besides, the equality above
occurs at most at a finite number of points, and the first eigenvalue
is positive.
\end{corollary}
\begin{proof}
Since $V$  is extremal, by the proposition above there exists a 
collection of eigenfunctions $(u_i)$, $i=1,\ldots,k$ , such that
$\sum_iu_i^2=1$. This means that the map
\begin{equation}
\label{map}
M\ni x\longmapsto (u_1(x),\ldots,u_k(x))\in S^{k-1}\subset\mathbf R^k
\end{equation}
is weakly harmonic, see~\cite{Hel96}. Since the dimension of $M$
equals $2$, the eigenfunctions are continuous, and by standard
regularity theory~\cite{LU68} the map given by~\eqref{map} is actually
$C^\infty$-smooth. Applying the Laplacian to the identity
$\sum_iu_i^2=1$, we further obtain the relation
\begin{equation}
\label{pot_f}
V=\lambda_1(V)-\sum_i\abs{\nabla u_i}^2.
\end{equation}
Thus, the potential $V$ is also $C^\infty$-smooth, and is bounded by
its first eigenvalue. The points where $\lambda_1(V)$ equals $V$
corresond to the branch points of the harmonic map~\eqref{map}; there
can be only finite number of these unless the harmonic map is
constant. The latter can not occur. For otherwise,
relation~\eqref{pot_f} together with the hypothesis $V\in
L^\infty_0(M)$ imply that both $V$ and $\lambda_1(V)$ vanish
identically.  Thus, $\lambda_1(V)$ becomes the first eigenvalue of the
Laplacian $(-\Delta_g)$, which is strictly positive -- a
contradiction. Finally, the positivity of $\lambda_1(V)$ follows by
integration of~\eqref{pot_f}.
\end{proof}

In sequel we freely identify collections of eigenfunctions $(u_i)$
such that $\sum_iu_i^2=1$ with harmonic maps into round spheres. For
basic properties and facts on the latter we refer to excellent
texts~\cite{EL83, Hel96}.

\subsection{Examples of extremal potentials}
Here we mention simplest examples of $\lambda_1$-extremal
potentials. We start with the case when $M$ is a standard round
sphere. The following proposition is a version of the theorem of
Hersch~\cite{H70}. 
\begin{prop}
\label{sphere}
Let $M$ be a $2$-sphere endowed with the standard round metric
$g$ of volume one. Then the maximal first eigenvalue of the
Schrodinger operator is equal to $8\pi$, and is achieved by the zero
potential. Further, any extremal potential on $M$ is maximal, and has
the form
$$
V(x)=64\pi^2-8\pi\abs{\nabla s}^2(x),\qquad x\in M,
$$
where $s:S^2\to S^2$ is a Mobious transformation and $\abs{\nabla s}$
stands for the Hilbert-Schmidt norm of its differential. In
particular, the space of maximal potentials is non-compact.
\end{prop}
The proof is outlined below; the key ingredient is the following
lemma, see~\cite{H70,LY82}. 
\begin{HL}
Let $y^i$, $i=1,2,3$, be coordinate functions in $\mathbf R^3$, and
$\phi:M\to S^2\subset\mathbf R^3$ be a conformal map to the unit
sphere centred at the origin. Then for any absolutely continuous
measure $\mu$ on $M$ there exists a conformal diffeomorphism $s:S^2\to
S^2$ such that
$$
\int_M(y^i\circ s\circ\phi)d\mu=0,\qquad\text{for all~ }i=1,2,3.
$$
\end{HL}
\begin{proof}[Proof of Prop.~\ref{sphere}]
First, we show that the zero potential is maximal. Its first
eigenvalue is the first Laplacian eigenvalue of the standard round
metric of volume one, and is equal to $8\pi$. By Hersch's lemma for
any $V\in L_0^\infty(M)$ the exists a conformal diffeomorphism
$s:S^2\to S^2$ such that the functions $(y^i\circ s)$, $i=1,2,3$, are
$L_2$-orthogonal to the ground state of the Schrodinger operator
$(-\Delta_g+V)$. Thus, by variational characterisation, we have
$$
\lambda_1(V)\int_M(y^i\circ s)^2\mathit{dVol}_g\leqslant\int_M
\abs{\nabla(y^i\circ s)}^2\mathit{dVol}_g+\int_MV(y^i\circ s)^2
\mathit{dVol}_g
$$
for any $i=1,2,3$. Since the volume of $g$ equals one, summing these
identities, we obtain 
$$
\lambda_1(V)\leqslant\sum_i\int_M\abs{\nabla(y^i\circ
  s)}^2\mathit{dVol}_g=\sum_i\int_M\abs{\nabla y^i}^2
\mathit{dVol}_g=8\pi;
$$
here the first equality holds by the conformal invariance of the
Dirichlet energy. Thus, $\lambda_1(V)\leqslant 8\pi$, and the zero
potential is, indeed, maximal.

Now we show that any extremal potential $V$ is, in fact, maximal. By
Prop.~\ref{props} there exists a collection of first eigenfunctions
$(u_i)$, $i=1,\ldots,k$, such that $\sum_iu_i^2=1$. By
Cor.~\ref{props:col} the potential $V$ is $C^\infty$-smooth, and the
result in~\cite{Ch76} says that the multiplicity of its first
eigenvalue is not greater than $3$. Thus, the harmonic
map~\eqref{map}, defined by eigenfunctions $(u_i)$, lies in the section
of the unit sphere by a subspace whose dimension is not greater than
$3$. In other words, this harmonic map is a map into the
$2$-dimensional unit sphere. As is known~\cite{EL83}, its energy 
$$
\int_M\sum_i\abs{\nabla u_i}^2\mathit{dVol}_g
$$
is an integer multiple of $8\pi$. By relation~\eqref{pot_f}, it
coincides with $\lambda_1(V)$, and by the discussion above  can be
either zero or $8\pi$. By Cor.~\ref{props:col}, the former can not
occur. Thus, the first eigenvalue $\lambda_1(V)$ equals to $8\pi$, and
the potential $V$ is maximal.

Finally, \cite[Cor~2.7]{El92} implies that any maximal potential $V$
on the standard $2$-sphere has the form
$$
V=64\pi^2-8\pi\sum_i\abs{\nabla(y^i\circ s)}^2,
$$
where $s$ is a Mobious transformation of $S^2$. The latter form a
non-compact group $\PSL(2,\mathbf{C})$, and the space of maximal
potentials can be identified with its topological quotient by the
equivalence relation
$$
s_1\sim s_2\qquad\text{iff}\qquad \abs{\nabla s_1}^2=\abs{\nabla s_2}^2.
$$
It is a straightforward calculation to see that the natural projection
onto $\PSL(2,\mathbf{C}) \backslash\sim$ is proper and, in particular,
the quotient $\PSL(2,\mathbf{C})\backslash\sim$ has to be non-compact.
\end{proof}

It is also straightforward to construct examples of extremal
potentials on tori. For example, by Prop.~\ref{props} for any flat
torus the zero potential is extremal for the first eigenvalue. 
Moreover, if $M$ is the Clifford torus (the quotient by the lattice
$\mathbf Z(1,0)\oplus\mathbf Z(0,1)$) or the equilateral torus (the
quotient by $\mathbf Z(1,0)\oplus\mathbf Z(1/2,\sqrt{3}/2)$), then the
zero potential is a unique global maximiser in $L^\infty_0(M)$;
see~\cite{El92,El00} for the details.

\section{Proof of Theorem~A: the alternative}
\label{alternative}
\subsection{The setup}
Let $V_n$, $n=1,2,\ldots,$ be a given sequence of extremal
potentials. Since the $\lambda_1(V_n)$'s are non-negative and,
by~\cite{El92,LY82}, uniformly bounded, without loss of generality we
can suppose that the sequence $\lambda_1(V_n)$ converges to a limit
$\lambda_*\geqslant 0$. By Prop.~\ref{props} for each $n\in\mathbf N$
there exists a finite collection of eigenfunctions $(u_{i,n})$,
$i=1,\ldots,m_n$, such that $\sum_iu^2_{i,n}=1$. Since the potentials
$V_n$'s are smooth, by the results in~\cite{Ch76} the multiplicities
of the $\lambda_1(V_n)$'s are uniformly bounded in terms of the genus
of $M$ only. Thus, after a selection of a subsequence, we may suppose
that for each $n\in \mathbf N$ there exists the same number of
eigenfunctions $(u_{i,n})$, $i=1,\ldots,m$, such that
$\sum_iu^2_{i,n}=1$. In other words, for each potential $V_n$ we have
a harmonic map
$$
M\ni x\longmapsto U_n(x)=(u_{i,n}(x))\in S^{m-1}\subset\mathbf R^m.
$$
As in the proof of Corollary~\ref{props:col}, we see that
\begin{equation}
\label{id}
\abs{\nabla U_n}^2=\sum_i\abs{\nabla u_{i,n}}^2=\lambda_1(V_n)-V_n.
\end{equation}
In particular, the energies
$$
E(U_n):=\int_M\abs{\nabla U_n}^2\mathit{dVol}_g
$$
of these harmonic maps are equal to $\lambda_1(V_n)$ and, hence, are
bounded. Now by the {\em bubbling convergence theorem} for harmonic
maps~\cite{SaU,Jo}, there exists a subsequence, also denoted by
$(U_n)$, which converges weakly in $W^{1,2}(M,S^{m-1})$ to a harmonic
map $U:M\to S^{m-1}$. Moreover, there exists a finite number of
``bubble points'' $\{x_1,\ldots, x_\ell\}\subset M$ such that the
$U_n$'s converge in $C^\infty$-topology on compact sets in
$M\backslash\{x_1,\ldots,x_\ell\}$, and the energy densities
$\abs{\nabla U_n}^2$ converge weakly in the sense of measures to
$\abs{\nabla U}^2$ plus a finite sum of Dirac measures:
$$
\abs{\nabla U_n}^2\rightharpoonup\abs{\nabla
  U}^2+\sum_je_j\delta_{x_j},
$$
where the constants $e_j>0$ correspond to the energies of the
so-called bubble spheres, see Sect.~\ref{bubble}-\ref{rems}.

Now we consider two cases when the energy density $\abs{\nabla U}^2$
of the limit map vanishes identically or not. In the former case we
obtain the claim~$(i)$; the case when $\abs{\nabla U}^2\not\equiv 0$
corresponds to the claim~$(ii)$.

\subsection{The case $\abs{\nabla U}^2\equiv 0$: concentration to a
  single point}
\label{concentrat}
First, we show that there is at least one ``bubble point''. For
otherwise, the harmonic maps $U_n$ converge in $C^\infty$-topology to
a constant map $U$. By relation~\eqref{id} we then conclude that the
potentials $V_n$ converge in $C^\infty$-topology to zero and so do
their first eigenvalues $\lambda_1(V_n)$. The latter implies that the
Laplacian $(-\Delta_g)$ has constant first eigenfunctions -- a
contradiction. Thus, the energy measures of the harmonic maps $U_n$
converge weakly to a sum of Dirac-measures,
$$
\abs{\nabla U_n}^2\rightharpoonup\mu=\sum_je_j\delta_{x_j}.
$$
Now we show that at most one delta-function can occur in the sum above.

Suppose the contrary. Then there are at least two ``bubble points''
$x_1$ and $x_2$. Denote by $\Omega_1$ and $\Omega_2$ their open
non-intersecting coordinate neighbourhoods that do contain any other
``bubble points''. Since a point has zero capacity, then for any
$\varepsilon>0$ there exist functions $\varphi_i\in
C^\infty_0(\Omega_i)$ such that $0\leqslant\varphi_i\leqslant 1$,
$$
\varphi_i=1\text{ in a neighbourhood of }x_i,\quad\text{and}\quad
\int_M\abs{\nabla\varphi_i}^2\mathit{dVol}_g<\varepsilon,\quad i=1,2.  
$$
Let $v_n$ be a $\lambda_0$-eigenfunction (ground state) of the
Schrodinger operator $(-\Delta_g+V_n)$. Further, let $\alpha_{1,n}$
and $\alpha_{2,n}$ be two sequences of real numbers such that the
linear combinations $\sum_i\alpha_{i,n}\varphi_i$ are $L_2$-orthogonal
to the $v_n$'s, and the sum of squares $\sum_i\alpha_i^2$ equals one
for any $n$. Without loss of generality, we may suppose that the
$\alpha_{i,n}$'s converge to some $\alpha_i$'s; the limit $\alpha_i$'s
clearly satisfy the relation $\sum_i\alpha_i^2=1$. Finally, denote by
$\psi_n$ and $\psi$ the functions $\sum_i\alpha_{i,n}\varphi_i$ and
$\sum_i\alpha_i\varphi_i$ respectively.

By construction, each function $\psi_n$ is $L_2$-orthogonal to $v_n$, and
by variational principle we have
\begin{equation}
\label{var_ineq}
\lambda_1(V_n)\int_M\psi_n^2\mathit{dVol}_g\leqslant\int_M\abs{\nabla\psi_n}^2
\mathit{dVol}_g+\int_MV_n\psi_n^2\mathit{dVol}_g.
\end{equation}
Since $V_n\rightharpoonup\lambda_*-\sum_je_j\delta_{x_j}$, then
passing to the limit, we obtain
$$
\lambda_*\int_M\psi^2\mathit{dVol}_g\leqslant\int_M\abs{\nabla\psi}^2
\mathit{dVol}_g+\lambda_*\int_M\psi^2\mathit{dVol}_g-\sum_i\alpha_i^2e_i.
$$
The last relation implies
$$
\sum_i\alpha_i^2e_i\leqslant\int_M\abs{\nabla\psi}^2\mathit{dVol}_g
\leqslant\varepsilon.
$$
Choosing $\varepsilon<\min\{e_i\}$, we obtain a contradiction. Thus,
the limit measure $\mu$ is one-point supported; $\mu=e\delta_x$ for
some $x\in M$.

Since the potentials $V_n$'s have zero mean-value, we conclude
from~\eqref{id} that the constant $e$ equals $\lambda_*$. Now for a
proof of the claim~$(i)$ it remains to show that $\lambda_*$ equals
$8\pi$. We explain this in Sect.~\ref{bubble}.

\subsection{The case $\abs{\nabla U}^2\not\equiv 0$: regularity of the
  limit measure}
Recall that by the bubbling convergence theorem for harmonic maps, the
energy densities $\abs{\nabla U_n}^2$ converge weakly to the measure
$$
\mu=\abs{\nabla U}^2+\sum_je_j\delta_{x_j};
$$
here we suppose that $U:M\to S^{m-1}$ is a non-trivial harmonic
map. First, the argument similar to the one in
Sect.~\ref{concentrat} shows that there is at most one ``bubble
point''. More precisely, if we suppose the contrary, then for a given
$\varepsilon>0$ we can choose the neighbourhoods $\Omega_1$ and
$\Omega_2$ such that
$$
\sum_{i}\int_{\Omega_i}\abs{\nabla U}^2\mathit{dVol}_g<\varepsilon.
$$
The potentials $V_n$ converge weakly to $\lambda_*-\mu$, and passing
to the limit in inequality~\eqref{var_ineq}, we obtain
$$
\sum_i\alpha_i^2e_i\leqslant\int_M\abs{\nabla\psi}^2\mathit{dVol}_g
+\sum_{i}\int_{\Omega_i}\abs{\nabla U}^2\mathit{dVol}_g
\leqslant 2\varepsilon.
$$
Now choosing $\varepsilon$ such that $2\varepsilon<\min\{e_i\}$, we
obtain a contradiction. Thus, the limit measure $\mu$ has the form
$\abs{\nabla U}^2+e\delta_x$. 

We claim that one ``bubble point'' can not occur also, and the limit
measure is absolutely continuous. Suppose the contrary. Let $\Omega$
be a coordinate ball centred at the ``bubble point'' $x$. Since the
capacity of a point is zero, then for any $\varepsilon>0$ there exists
a function $\varphi\in C^\infty_0(\Omega)$ such that
$0\leqslant\varphi\leqslant 1$,
$$
\varphi=1\text{ in a neighbourhood of }x,\quad\text{and}\quad
\int_M\abs{\nabla\varphi}^2\mathit{dVol}_g<\varepsilon.
$$
As in Sect.~\ref{concentrat} by $v_n$ we denote positive ground states
of the Schrodinger operators $(-\Delta_g+V_n)$; we assume that their
$L_1$-norms are equal to one. Consider the sequence
$$
0<\alpha_n=\int_M\varphi\cdot v_n\mathit{dVol}_g\leqslant 1;
$$
without loss of generality, we may assume that the $\alpha_n$'s
converge to some limit $\alpha\geqslant 0$. By $\psi_n$ we denote the
functions $(\varphi-\alpha_n)$, and by $\psi$ the function
$(\varphi-\alpha)$. Since each $\psi_n$ is $L_2$-orthogonal to
$v_n$, by variational principle we have
$$
\lambda_1(V_n)\int_M\psi^2_n\mathit{dVol}_g\leqslant\int_M
\abs{\nabla\psi_n}^2\mathit{dVol}_g+\int_MV_n\psi_n^2\mathit{dVol}_g.
$$
Since $V_n\rightharpoonup\lambda_*-\abs{\nabla U}^2-e\delta_x$, then
passing to the limit and making elementary transformations, we obtain
$$
\int_M\abs{\nabla U}^2\psi^2\mathit{dVol}_g+e(1-\alpha)^2\leqslant
\int_M\abs{\nabla\psi}^2\mathit{dVol}_g.
$$
The last relation implies
$$
\alpha^2\int_{M\backslash\Omega}\abs{\nabla U}^2\mathit{dVol}_g+
(1-\alpha)^2e\leqslant\int_M\abs{\nabla\varphi}^2\mathit{dVol}_g<\varepsilon
$$
By elementary analysis, the left-hand side is bounded below by
$$
0<\left(e\int_{M\backslash\Omega}\abs{\nabla U}^2\mathit{dVol}_g\right)/
\left(e+\int_{M\backslash\Omega}\abs{\nabla U}^2\mathit{dVol}_g\right).
$$
This yields a contradiction, since $\varepsilon>0$ is arbitrary.

Thus, we see that no bubbling can occur, and the harmonic maps $U_n$
converge in $C^\infty$-topology to the harmonic map $U$. Further, by
relation~\eqref{id} the extremal potentials $V_n$ also converge in
$C^\infty$-topology to the potential
$$
V=\lambda_*-\abs{\nabla U}^2,\qquad V\in L^\infty_0(M).
$$
By standard perturbation theory~\cite{Kato}, the eigenvalues
$\lambda_1(V_n)$ has to converge to $\lambda_1(V)$, and we conclude
that the first eigenvalue $\lambda_1(V)$ coincides with $\lambda_*$.
Further, we see that the components $u_i$, $i=1,\ldots,m$, of the
harmonic map $U$ are first eigenfunctions of the Schrodinger operator
$(-\Delta_g+V)$. Finally, since $\sum_iu_i^2=1$, Prop.~\ref{props}
implies that the potential $V$ is extremal.

\section{Proof of Theorem~A: the eigenvalue of the bubble.}
\label{bubble}
For a proof of Theorem~A it remains to show that the hypotheses
$V_n\rightharpoonup\lambda_*(1-\delta_x)$, and
$\lambda_1(V_n)\to\lambda_*$ as $n\to +\infty$ imply that $\lambda_*$
has to be equal to $8\pi$. This is the content of the present
section. First, we prove the estimate $\lambda_*\leqslant 8\pi$ for
concentrating sequences of not necessarily extremal potentials.
To get the lower bound we study the Schrodinger equation on the bubble
sphere obtained as the limit equation under convergence of renormalised
eigenfunctions.

\subsection{General upper bound: $\lambda_*\leqslant 8\pi$}
The following lemma gives an estimate for arbitrary concentrating
sequences of potentials; cf.~\cite[p.888-889]{Na96}.
\begin{lemma}
\label{upper_b}
Let $M$ be a closed surface endowed with a Riemannian metric $g$, and
$V_n\in L^\infty_0(M)$, $n=1,2,\ldots$, be a sequence such that
$V_n\rightharpoonup\lambda_*(1-\delta_x)$, and $\lambda_1(V_n)
\to\lambda_*$ as $n\to +\infty$. Then the number $\lambda_*$ is not
greater than $8\pi$. 
\end{lemma}
\begin{proof}
Let $\Omega$ be an open coordinate ball around $x\in M$ on which the
metric $g$ is conformally Euclidean, and let
$$
\phi:\Omega\longrightarrow S^2\subset\mathbf R^3
$$
be a conformal map into the unit sphere in $\mathbf R^3$. Since a
point on the Euclidean plane has zero capacity, then for any
$\varepsilon>0$ there exists a function $\psi\in C^\infty_0(\Omega)$
such that $0\leqslant\psi\leqslant 1$,
$$
\psi=1\text{ in a neighbourhood of }x,\quad\text{and}\quad
\int_M\abs{\nabla\psi}^2\mathit{dVol}_g<\varepsilon.
$$
As above by $v_n$ we denote a positive ground state of the Schrodinger
operator $(-\Delta_g+V_n)$. By Hersch's lemma, Sect.~\ref{prems},
there exists a conformal transformation $s_n:S^2\to S^2$ such that
$$
\int_M\psi(y^i\circ s_n\circ\phi)v_n\mathit{dVol}_g=0\qquad
\text{for any}\quad i=1,2,3,
$$
where $(y^i)$ are coordinate functions in $\mathbf R^3$. In other
words, each function $\varphi^i_n=\psi(y^i\circ s_n\circ\phi)$ is
$L_2$-orthogonal to $v_n$, and by variational principle we have
$$
\lambda_1(V_n)\int_M(\varphi_n^i)^2\mathit{dVol}_g\leqslant\int_M
\abs{\nabla\varphi^i_n}^2\mathit{dVol}_g+\int_MV_n(\varphi^i_n)^2
\mathit{dVol}_g,
$$
for any $i=1,2,3$. Summing with respect to $i$, we obtain
\begin{equation}
\label{ineq}
\lambda_1(V_n)\int_M\psi^2\mathit{dVol}_g\leqslant\sum_i\int_M
\abs{\nabla\varphi^i_n}^2\mathit{dVol_g}+\int_MV_n\psi^2\mathit{dVol}_g.
\end{equation}
Now we estimate the first term on the right-hand side
\begin{multline*}
\sum_i\int_M\abs{\nabla\varphi^i_n}^2\mathit{dVol}_g\leqslant\sum_i
\int_M\psi^2\abs{\nabla(y^i\circ s_n\circ\phi)}^2\mathit{dVol}_g\\
+2\sum_i\int_M\psi\abs{\nabla(y^i\circ s_n\circ\phi)}\abs{\nabla\psi}
\mathit{dVol}_g+\int_M\abs{\nabla\psi}^2\mathit{dVol}_g.
\end{multline*}
The first sum on the right-hand side can be further estimated by
the quantity
$$
\sum_i\int_\Omega\abs{\nabla(y^i\circ s_n\circ\phi)}^2
\mathit{dVol}_g\leqslant\sum_i\int_{S^2}\abs{\nabla(y^i\circ s_n)}^2
\mathit{dVol}_{S^2}=8\pi;
$$
here we used the conformal invariance of the Dirichlet energy, which
in particular implies that the energy of a conformal diffeomorphism of
$S^2$ equals $8\pi$. Similarly the second sum is not greater that
\begin{multline*}
2\sum_i\int_\Omega\abs{\nabla(y^i\circ s_n\circ\phi)}\abs{\nabla\psi}
\mathit{dVol}_g\leqslant 2\varepsilon^{1/2}\sum_i\left(\int_\Omega
\abs{\nabla(y^i\circ s_n\circ\phi)}^2\mathit{dVol}_g\right)^{1/2}\\
\leqslant 10\pi^{1/2}\varepsilon^{1/2}.
\end{multline*}
Using these two estimates and the fact that the Dirichlet energy of
$\psi$ is less than $\varepsilon$, we obtain
$$
\sum_i\int_M\abs{\nabla\varphi^i_n}^2\mathit{dVol}_g\leqslant
8\pi+10\pi^{1/2}\varepsilon^{1/2}+\varepsilon.
$$
Combining the last inequality with the one in~\eqref{ineq}, and
passing to the limit as $n\to+\infty$, we arrive at the following
relation
$$
\lambda_*\leqslant 8\pi+10\pi^{1/2}\varepsilon^{1/2}+\varepsilon.
$$
Since $\varepsilon>0$ is arbitrary, we conclude that
$\lambda_*\leqslant 8\pi$.
\end{proof}

\subsection{The Schrodinger equation on the bubble sphere}
To obtain the lower bound $\lambda_*\geqslant 8\pi$, we study a
certain Schrodinger equation on the so-called bubble
sphere. Bubble spheres appear as  natural ``limit objects'' of
sequences of renormalised harmonic maps, describing the behaviour of
sequences near bubble points; see~\cite{SaU,Pa96}. The construction of 
a bubble sphere below uses a slightly non-standard renormalisation
procedure that is more suitable in our context.

We start with a sequence of harmonic maps
$$
M\ni x\longmapsto U_n(x)=(u_{i,n}(x))\in S^{m-1}\subset\mathbf R^m
$$
whose coordinates $u_{i,n}$ are first eigenfunctions of the
Schrodinger operator $(-\Delta_g+V_n)$. We consider the case when the
concentration occurs -- the sequence $\abs{\nabla U_n}^2$ converges
weakly to the one-point supported singular measure
$\lambda_*\delta_x$, see Sect.~\ref{alternative}. In particular,
$$
\Lambda_n=\max_{x\in M}\abs{\nabla U_n}^2(x)\to +\infty
\qquad\text{as }n\to+\infty.
$$
Let $x_n\in M$ be a point where the maximum of $\abs{\nabla U_n}^2(x)$
is achieved; without loss of generality, we can assume that the
$x_n$'s converge to a point $x_*\in M$. Let $\Omega$ be a chart ball
centred at $x_*$; we suppose that the metric $g$ is conformally
Euclidean on $\Omega$ and $g_{ij}(x_*)=\delta_{ij}$. For a sufficiently
large $n$, the mapping 
$$
\phi_n:D_n=\left\{x\in\mathbf R^2:\abs{x}<\sqrt{\Lambda_n}\right\}\to
\Omega,\qquad x\to x/\Lambda_n+x_n,
$$
is well-defined. We endow the ball $D_n$ with a Riemannian metric
$(g_n)_{ij}=g_{ij}\circ\phi_n$; equivalently, the $g_n$ equals
$\Lambda_n^2(\phi_n^*g)$. Consider the functions ${\bar u}_{i,n}=
u_{i,n}\circ\phi_n$ on each $D_n$; they satisfy the equations
\begin{equation}
\label{eq1}
-\Delta_{g_n}\bar u_{i,n}=\frac{1}{\Lambda_n^2}
\left(\lambda_1(V_n)-\bar V_n\right)\bar u_{i,n},
\end{equation}
where $\bar V_n=V_n\circ\phi_n$. Applying the Laplacian $\Delta_{g_n}$
to the identity $\sum_i\bar u_{i,n}^2=1$, we conclude that the maps
$$
D_n\ni x\longmapsto\bar U_n(x)=(\bar u_{i,n}(x))\in
S^{m-1}\subset\mathbf R^m
$$
are harmonic and satisfy the relations
\begin{equation}
\label{eq2}
\abs{\nabla\bar U_n}^2_{g_n}=\frac{1}{\Lambda_n^2}\left(
\lambda_1(V_n)-\bar V_n\right).
\end{equation}
By the definition of the $\phi_n$'s, we also have 
\begin{equation}
\label{eq3}
\abs{\nabla\bar U_n}_{g_n}(x)\leqslant 1,\qquad\text{and}\qquad
\abs{\nabla\bar U_n}_{g_n}(0)=1.
\end{equation}
Since the metrics $g_n$ converge to the Euclidean metric on $\mathbf
R^2$, the first inequality above together with standard Schauder
estimates, see~\cite{LU68}, imply that the maps $\bar U_n$ converge in
$C^\infty$-topology to a harmonic map $\bar U:\mathbf R^2\to\Omega$ on
each compact subset of $\mathbf R^2$. Finally, since the Dirichlet
energy is conformally invariant, it is straightforward to show that
\begin{equation}
\label{eq4}
\lim\sup\int_{D_n}\abs{\nabla\bar U_n}^2_{g_n}\mathit{dVol}_{g_n}
\leqslant\lambda_*\qquad\text{as }n\to+\infty.
\end{equation}
Besides, if $x_*$ does not coincide with the bubble point $x$, the
$\lim\sup$ on the left-hand side above vanishes.

Identifying $\mathbf R^2$ with $S^2\backslash\{p\}$ via the
stereographic projection, we can view $\bar U=(\bar u_i)$,
$i=1,\ldots,m$, as a harmonic map from $S^2\backslash\{p\}\to S^{m-1}$,
where the sphere $S^2$ is endowed with  the standard round metric
$g_s$. Using the conformal invariance of energy again, we conclude from
inequality~\eqref{eq4} that the map $\bar U$ has finite energy,
$E(\bar U)\leqslant\lambda_*$. Hence, by~\cite{SaU} its singularity at
$p$ is removable -- the map $\bar U$ extends to a smooth harmonic map
$S^2\to S^{m-1}$. By the second relation in~\eqref{eq3}, the map $\bar
U$ is non-constant, and its energy $E(\bar U)$ is strictly
positive. In particular, we conclude that the point $x_*$ coincides
with the bubble point $x$. Denote the energy of $\bar U$ by
$\bar\lambda$ and define the potential $\bar V$ on the sphere $S^2$ by
the formula
$$
\bar V=\bar\lambda-\abs{\nabla\bar U}^2_{g_s}.
$$
Clearly, it belongs to the space $L_0^\infty(S^2)$, and by
relation~\eqref{eq2} we have
$$
\frac{\kappa}{\Lambda_n^2}\left(\lambda_1(V_n)-\bar V_n\right)\to
\bar\lambda-\bar V
$$
in $C^\infty$-topology on compact sets in $S^2\backslash\{p\}$. Here
$\kappa$ stands for the conformal factor between the Euclidean metric
on $S^2\backslash\{p\}$ and the standard metric $g_s$ on $S^2$. Since
the Laplacian is conformally invariant in dimension two, then passing
to the limit in equation~\eqref{eq1}, we obtain
$$
(-\Delta_{g_s}+\bar V)\bar u_i=\bar\lambda\bar u_i,
\qquad\text{where }i=1,\ldots,m.
$$
Thus, we see that $\bar\lambda$ is an eigenvalue for the Schrodinger
operator $(-\Delta_{g_s}+\bar V)$ on the sphere, and the $\bar u_i$'s
are its eigenfunctions.
\begin{lemma}
\label{mult}
The eigenfunctions $\bar u_i$, $i=1,\ldots,m$, span a vector space
whose dimension is at most $3$.
\end{lemma}

By Lemma~\ref{mult}, we see that the harmonic map $\bar U$, defined by
eigenfunctions $(\bar u_i)$, lies in the section of the unit sphere by
a subspace whose dimension is not greater than $3$. In other words the
harmonic map $\bar U$ is actually a harmonic map into the
$2$-dimensional unit sphere. Hence, its energy is an integer multiple
of $8\pi$ and, since $\bar\lambda>0$, we conclude that $\bar\lambda$
has to be at least $8\pi$. On the other hand, we have $\bar\lambda
\leqslant\lambda_*$ and, by Lemma~\ref{upper_b}, the latter is not
greater than $8\pi$. Thus, we obtain that $\lambda_*$ equals $8\pi$,
finishing the proof of Theorem~A. The rest of this section is devoted
to the proof of Lemma~\ref{mult}.

\subsection{Proof of Lemma~\ref{mult}}
To prove the lemma we analyse the structure of the nodal set of the
eigenfunctions $\bar u_i$. Following the arguments of Cheng~\cite{Ch76},
this allows to bound the vanishing order of the $\bar u_i$'s at each
nodal point, and hence to estimate the dimension of
$\mathit{Span}(\bar u_i)$.

First, since the $u_{i,n}$'s are first eigenfunctions of the
Schrodinger operator $(-\Delta_g+V_n)$, then each of them
changes sign. Moreover, by the results in~\cite{Ch76} the nodal set
$u_{i,n}^{-1}(0)$ is an immersed circle in $M$, and the complement
$M\backslash u_{i,n}^{-1}(0)$ has exactly two connected components,
called {\it nodal domains}. We claim that the bubble point $x$ belongs
to the closure of the set
$$
\bigcup_n u_{i,n}^{-1}(0)\qquad\text{for every}\qquad i=1,\ldots,m.
$$
Indeed, for otherwise there exists a neighbourhood of $x$ which
belongs to the nodal domain of $u_{i,n}$ for all sufficiently large
$n$. This, in turn, implies that the limit map $\bar u_i$ does not
change sign on the bubble sphere, and hence $\bar\lambda$ has to be a
zero eigenvalue for the corresponding Schrodinger operator. The latter
clearly contradicts to the fact that $\bar\lambda$ is positive.

A similar analysis yields that each limit eigenfunction $\bar u_i$ on
the bubble sphere has exactly two nodal domains; their nodal lines are
limits of renormalised arcs on the nodal lines of the $u_{i,n}$'s. Now
the structure theorem in~\cite{Ch76} implies that any point on the
nodal line $\bar u_i^{-1}(0)$ has vanishing order at most one. In more
detail, the nodal set near a critical point with vanishing order $k$
is diffeomorphic to the nodal set of a spherical harmonic of order $k$
in $\mathbf R^2$, which consists of $k$ straight lines passing through
the origin, see~\cite[Lem.~3.3]{Ch76}. Therefore, if $k$ is greater
than one, then by~\cite[Lem.~3.1]{Ch76} the set $S^2\backslash\bar
u_i^{-1}(0)$ has at least $3$ connected components -- a contradiction.

The same analysis equally applies to a non-trivial linear combination
of the $\bar u_i$'s, and we conclude that any point on its
nodal line also has a vanishing order at most one. Now
following~\cite[Th.~3.4]{Ch76}, we show that the dimension of
$\mathit{Span} (\bar u_i)$, $i=1,\ldots,m$, is not greater than
$3$. Suppose the contrary. Then for any $z\in S^2$ the map
$$
\mathit{Span} (\bar u_i)\ni v\longmapsto (v(z),\nabla v(z))\in\mathbf
R^3
$$
has a non-trivial kernel -- there exists a non-trivial linear
combination of the $\bar u_i$'s that vanishes at $z$ together with its
first derivatives. Thus, the vanishing order at $z$ is greater than
one -- a contradiction.\qed

\section{Final remarks}
\label{rems}

\noindent
{\it 1.~} The proof of Theorem~B is based on the version of the
bubbling convergence theorem for harmonic maps with a variable metric
on the domain surface, see~\cite[Lem.~1.2]{Pa96}. All our arguments in
Sect.~\ref{alternative} and~\ref{bubble} admit obvious adjustments to
cover this case also.

\medskip
\noindent
{\it 2.~} One can analyse the concentration of extremal
potentials from the point of view of the bubble tree convergence of
harmonic maps, as described in~\cite{Pa96}. (The latter is based on a
different renormalisation at the bubble point than the one used in
Sect.~\ref{bubble}). More precisely, one can show that when
$V_n\rightharpoonup e(1-\delta_x)$ the corresponding harmonic maps
$U_n$, given by eigenfunctions, converge to a constant harmonic map
with only one bubble attached at the point $x$; in other words, no
``secondary'' bubbles appear. Finally, mention that the equality
$\bar\lambda=\lambda_*$, obtained in Sect.~\ref{bubble}, reflects the
``no energy loss at the neck'' phenomenon.

\medskip
\noindent
{\it 3.~} It is extremely interesting to understand under what
hypotheses analogous concentration compactness properties hold for
more general (for example, maximising) sequences of potentials. 
This question is motivated by the existence problem for maximal (or
extremal) potentials, and has strong links with isoperimetric
inequalities for eigenvalues, see~\cite{LY82,Na96}.


{\small
\begin{thebibliography}{99}
\addcontentsline{toc}{section}{References}

\bibitem{Ch76} Cheng,~S.~Y. {\em Eigenfunctions and nodal sets.}
Comment. Math. Helv. {\bf 51} (1976), 43--55. 

\bibitem{EL83} Eells,~J., Lemaire,~L. {\em Selected topics in harmonic
maps.} CBMS Regional Conference Series in Mathematics, 50. AMS,
Providence, RI, 1983. v+85 pp.

\bibitem{El92} El Soufi,~A. Ilias,~S. {\em Majoration de la seconde
valeur propre d'un op\'erateur de Schr\"odinger sur une vari\'et\'e
compacte et applications.} J. Funct. Anal. {\bf 103} (1992), 294--316.

\bibitem{El00} El Soufi,~A. Ilias,~S. {\em Riemannian manifolds
admitting isometric immersions by their first eigenfunctions} Pacific
J. Math. {\bf 195} (2000), 91--99.

\bibitem{Hel96} H\'elein, F. {\em Harmonic maps, conservation laws and
moving frames.} Translated from the 1996 French original. Second
edition. Cambridge Tracts in Mathematics, 150. Cambridge University
Press, Cambridge, 2002. xxvi+264 pp. 

\bibitem{H70} Hersch,~J. {\em Quatre propri\'et\'es isop\'erim\'etriques de
membranes sph\'eriques homog\`enes.} C. R. Acad. Sci. Paris S\'er. A-B
{\bf 270} (1970), A1645--A1648.

\bibitem{Jo} Jost,~J. {\em Two-dimensional geometric variational
problems.} Pure and Applied Mathematics. John Wiley \& Sons, Ltd.,
Chichester, 1991. x+236 pp.

\bibitem{Kato} Kato,~T. {\em Perturbation theory for linear
operators.} Second edition. Grundlehren der Mathematischen
Wissenschaften, Band 132. Springer-Verlag, Berlin-New York,
1976. xxi+619 pp.


\bibitem{LU68} Ladyzhenskaya,~O. Ural'tseva,~N. {\em Linear and
quasilinear elliptic equations.} Academic Press, New York-London,
1968, xviii+495 pp. 

\bibitem{LY82} Li,~P., Yau,~S.-T. {\em A new conformal invariant and
its applications to the Willmore conjecture and the first eigenvalue
of compact surfaces.} Invent. Math. {\bf 69} (1982), 269--291. 

\bibitem{Na96} Nadirashvili,~N. {\em Berger's isoperimetric problem
and minimal immersions of surfaces.} Geom. Funct. Anal. {\bf 6}
(1996), 877--897. 

\bibitem{Pa96} Parker,~T.~H. {\em Bubble tree convergence for harmonic
maps.}  J. Differential Geom. {\bf 44} (1996), 595--633.

\bibitem{SaU} Sacks,~J., Uhlenbeck,~K. {\em The existence of minimal
immersions of $2$-spheres.} Ann. of Math. (2) {\bf 113} (1981), 1--24. 

\end{thebibliography}
}
\end{document}